\documentclass[10pt]{amsart}
\usepackage{amsmath,amssymb,amsthm,a4wide}
\usepackage[all]{xy}
\usepackage{graphicx}
\newtheorem{theorem}{Theorem}[section]    
\newtheorem{lemma}[theorem]{Lemma}          
\newtheorem{proposition}[theorem]{Proposition}  

\newtheorem{corollary}[theorem]{Corollary} 

\theoremstyle{definition}

\newtheorem{remark}[theorem]{Remark}

\newtheorem{example}[theorem]{Example}

 \makeatletter
    
    \@addtoreset{equation}{section}
  \makeatother
\def\co{\colon \thinspace}

\newcommand{\GL}{\textrm{GL}}
\newcommand{\Z}{\mathbb{Z}}
\newcommand{\R}{\mathbb{R}}

\newcommand{\C}{\mathbb{C}}
\newcommand{\bL}{\mathbb{L}}

\newcommand{\sltwo}{\mathfrak{sl}_{2}}

\newcommand{\mQ}{\mathcal{Q}}

\newcommand{\tC}{\widetilde{C}}
\newcommand{\bF}{\mathbb{F}}
\newcommand{\be}{\mathbf{e}}

\newcommand{\half}{\frac{1}{2}}

\DeclareMathOperator{\tr}{tr}
\DeclareMathOperator{\Sym}{Sym}

\def\diaCrossP{\unitlength.1em
  \begin{minipage}{10\unitlength}
    \begin{picture}(10,10)
      \put(0,0){\vector(1,1){10}}
      \qbezier(10,0)(10,0)(7,3)
      \qbezier(3,7)(0,10)(0,10)
      \put(0,10){\vector(-1,1){0}}
    \end{picture}
  \end{minipage}
}

\def\diaCrossN{\unitlength.1em
  \begin{minipage}{10\unitlength}
    \begin{picture}(10,10)
      \put(10,0){\vector(-1,1){10}}
      \qbezier(0,0)(0,0)(3,3)
      \qbezier(7,7)(10,10)(10,10)
      \put(10,10){\vector(1,1){0}}
    \end{picture}
  \end{minipage}
}

\def\diasolve{\unitlength.1em
  \begin{minipage}{10\unitlength}
    \begin{picture}(10,10)
      \put(0,10){\vector(-1,1){0}}
      \qbezier(0,0)(5,5)(0,10)
      \qbezier(10,0)(5,5)(10,10)
      \put(10,10){\vector(1,1){0}}
    \end{picture}
  \end{minipage}
}

\title[Topological formula of the colored Jones polynomials]{Topological formula of the loop expansion of the colored Jones polynomials}
\author{Tetsuya Ito}
\address{Research Institute for Mathematical Sciences, Kyoto university, Kyoto, 606-8502, Japan}
\email{tetitoh@kurims.kyoto-u.ac.jp}
\urladdr{http://www.kurims.kyoto-u.ac.jp/~tetitoh/}

\begin{document}

\begin{abstract} 
We give a topological formula of the loop expansion of the colored Jones polynomials by using identification of generic quantum $\sltwo$ representation with homological representations. This gives a direct topological proof of the Melvin-Morton-Rozansky conjecture, and a connection between entropy of braids and quantum representations.
\end{abstract}
\maketitle

\section{Introduction}

For $\alpha \in \{2,3,4,\ldots\}$ and an oriented knot $K$ in $S^{3}$, let $J_{K,\alpha}(q) \in \Z[q,q^{-1}]$ be the $\alpha$-colored Jones polynomial of $K$ normalized so that $J_{\textsf{Unknot}, \alpha}(q) = 1$. As Melvin-Morton proved \cite{mm}, by putting $q=e^{\hbar}$, the colored Jones polynomials can be expanded as a power series of two independent variables $\hbar\alpha$ and $\hbar$, as
\[
 J_{K,\alpha}(e^{\hbar}) = \sum_{i=0}^{\infty} D^{(i)}(\hbar \alpha) \hbar^{i}  = \sum_{i=0}^{\infty}\left(\sum_{k=0}^{\infty} d^{(i)}_{k}(\hbar\alpha)^{k}\right) \hbar^{i}. \]
Further, we put $z = e^{  \hbar \alpha}$ and write the colored Jones polynomials as a function on $\hbar$ and $z$,
\[ CJ_{K}(z,\hbar) = J_{K,\alpha}(e^{\hbar})= \sum_{i=0}^{\infty} V_{K}^{(i)}(z) \hbar^{i}. \]
We call $CJ_{K}$ the \emph{colored Jones function} or, the \emph{loop expansion of quantum $\sltwo$-invariant} since it coincides with the $\sltwo$ weight system reduction of the loop expansion of the Kontsevich invariant. In particular, the $i$-th coefficient $V^{(i)}(z)$ corresponds to the $(i+1)$-st loop part of the loop expansion of the Kontsevich invariant \cite{oht1}.

Let $\nabla_K(z)$ be the Conway polynomial of $K$, characterized by the skein relation 
\[ \nabla_{\diaCrossP}(z) - \nabla_{\diaCrossN}(z) = z \nabla_{\diasolve}(z),\ \ \ \nabla_{\sf Unknot}(z) = 1 \]
and let $\Delta_{K}(z) = \nabla_{K}(z^{\frac{1}{2}}-z^{-\frac{1}{2}})$ be the Alexander-Conway polynomial. The Alexander-Conway polynomial appears as one of the basic building block of $CJ_{K}$. The Melvin-Morton-Ronzansky conjecture \cite{mm} (MMR conjecture, in short), proven in \cite{bg}, states that the $V^{(0)}(z)$ is equal to $\Delta_{K}(z)^{-1}$. More generally, $V^{(i)}(z)$ is a rational function whose denominator is $\Delta_{K}(z)^{2i+1}$ \cite{ro1}.

In a theory of quantum invariants, the appearance of the Alexander-Conway polynomial is well-understood. The aforementioned rationality of $V^{(i)}(z)$ follows from Ronzansky's rationality conjecture \cite{ro2} of the loop expansion of Kontsevich invariant, proven in \cite{kr}: The Aarhus integral computation of Kontsevich (or, LMO) invariant, based on a surgery presentation of knots, provides the desired rationality (see \cite[Section 1.2]{gk} for a brief summary of Kricker's argument).

The clasper surgery \cite{ha} explains a geometric connection between the loop expansion and infinite cyclic covering \cite{gr}. A null-clasper, a clasper with null-homologous leaves in the knot complement, lifts to a clasper in the infinite cyclic covering, and the loop expansion nicely behaves under the clasper surgery along null-claspers. Thus schematically speaking, the loop expansion provides a $\Z$-equivariant quantum invariants \cite{gr} (for example, the 2-loop part can be interpreted as the $\Z$-equivariant Casson invariant, as discussed in \cite{oht2}), so it is not surprising that the Alexander-Conway polynomial appears in the loop expansion.

Nevertheless, it is still mysterious why the Alexander polynomial appears in such a particular and direct form. Even for the MMR conjecture, the simplest and the most fundamental rationality result, the situation is not so good as we want. In a known proof, one uses quantum-invariant-like treatment of the Alexander-Conway polynomial such as, state-sum, R-matrix, or weight systems so its topological content is often indirect.

In this paper, we give a topological formula of $CJ_{K}$ by using homological braid group representations (Theorem \ref{theorem:main}). Our starting point is a recent result in \cite{i,koh2} that identifies certain homological representations introduced by Lawrence \cite{law} with generic $U_{q}(\sltwo)$ representations. This allows us to translate a construction of the colored Jones function in terms of corresponding homological representations. Also, in Section \ref{sec:ent} we discuss a connection among the entropy of braids, quantum representations, and the growth of quantum $\sltwo$ invariants inspired from topological point of view.

One may notice that our approach is similar to Lawrence-Bigelow's approach of Jones polynomial \cite{big,law2}, but there are several critical differences:
We give a formula of the loop expansion $CJ_{K}$ but do not provide a formula of each individual colored Jones polynomials. Our formula uses closed braid representatives, whereas Lawrence-Bigelow description uses plat representatives and intersection products.

It should be emphasized that quantum representations coming from finite dimensional $U_{q}(\sltwo)$-module is \emph{not} identified with homological representation. This is the reason why we do not have a direct topological formula of usual colored Jones polynomials.

Our topological description leads to several insights.
First, the MMR conjecture is now obtained as a direct consequence of our topological formula. By putting $\hbar =0$, topological considerations show that the homological representations is equal to the symmetric powers of the reduced Burau representation, so they naturally lead to the Alexander-Conway polynomial.
Second, our formula gives a new and direct way to calculate $CJ_{K}(z,\hbar)$ without knowing or computing individual colored Jones polynomial $J_{K,\alpha}(q)$ or appealing surgery presentation of knots, although a general calculation is still difficult.

\section*{Acknowledgments}
The author was partially supported by JSPS Grant-in-Aid for Research Activity start-up, Grant Number 25887030.
He would like to thank Tomotada Ohtsuki, Jun Murakami and Hitoshi Murakami for stimulating discussion and comments.

\section{A topological description of generic quantum $\sltwo$ representation}

In this section we review the result in \cite{i} that identifies a generic quantum $\sltwo$ representation given in \cite{jk} with Lawrence's homological representation and some additional arguments to treat non-generic case. 

Throughout the paper, we use the following notations and conventions.
The $q$-numbers, $q$-factorials, and $q$-binomial coefficients are defined by
\[ [n]_{q} = \frac{q^{\frac{n}{2}}-q^{-\frac{n}{2}}}{q^{\half}-q^{-\half}}, \ \ [n]_{q} ! = [n]_{q}[n-1]_{q} \cdots [2]_{q}[1]_{q} , \ \  \left[ \!\!\begin{array}{c} n \\ j \end{array} \!\! \right]_{q} = \frac{[n]_{q}!}{[n-j]_{q}![j]_{q}!}, \]
respectively. This convention is different from one in \cite{i,jk}. The quantum parameter $q$ in this paper corresponds to $q^{2}$ in \cite{i,jk}.
We always assume that the braid group $B_{n}$ is acting from left.

Let $R$ be a commutative ring.
For $R$-modules (resp. $R B_{n}$-modules) $V$ and $W$, we denote $V \cong_{\mQ} W$ if they are isomorphic over the quotient field $\mQ$ of $R$. Namely, $C \cong_{\mQ} W$ implies $V \otimes_{R} \mQ$ and $W \otimes_{R} \mQ$ are isomorphic as $\mQ$-modules (resp. $\mQ B_{n}$-module).

For a subring $R \subset \C$, let $R[x^{\pm  1}]$ be the Laurent polynomial ring, and for an $R[x^{\pm 1}]$-module $V$ and $c \in \C$, we denote the specialization of the variable $x$ to complex parameter $c$ by $V|_{x=c}$.

\subsection{Generic quantum representation}

Let $\C[[\hbar]]$ be the algebra of the complex formal power series in one variable $\hbar$, and we put $q=e^{\hbar}$, as usual. 
A quantum enveloping algebra $U_{\hbar}(\sltwo)$ is a topological Hopf algebra over $\C[[\hbar]]$ generated by $H,E,F$ subjected to the relations
\begin{gather*}
\begin{cases}
[H,E] = 2E, \;\;\;[H,F] = -2F , \\
{}\displaystyle [E,F] = \frac{\sinh(\frac{\hbar H}{2})}{\sinh(\frac{\hbar}{2})} = \frac{e^{\frac{\hbar H}{2}} - e^{- \frac{\hbar H}{2}}}{e^{\frac{\hbar}{2}}- e^{-\frac{\hbar}{2}}}
\end{cases}
\end{gather*}

$U_{\hbar}(\sltwo)$ is a quasi-triangular topological Hopf algebra and a \emph{universal $R$-matrix} $\mathcal{R} \in U_{\hbar}(\sltwo) \otimes U_{\hbar}(\sltwo)$  
is given by
\begin{equation}
\label{eqn:univR}
\mathcal{R} = e^{\frac{\hbar}{4} (H \otimes H)} \left( \sum_{n=0}^{\infty} q^{ \frac{n(n-1)}{4}} \frac{(q^{\frac{1}{2}}-q^{-\frac{1}{2}})^{n}} {[n]_{q} !} E^{n} \otimes F^{n} \right).
\end{equation}
(Strictly speaking, here we need to use the topological tensor product $\widetilde{\otimes}$, the $\hbar$-adic completion of $U_{\hbar}(\sltwo) \otimes U_{\hbar}(\sltwo)$. To make notation simple, in the rest of the paper $\otimes$ should be regarded as the topological tensor product, if we should do so.) 

For $\lambda \in \C^{*}$, let $V_{\lambda}$ be the Verma module of highest weight $\lambda$, a topologically free $U_{\hbar}(\sltwo)$-module generated by a highest weight vector $v_{0}$ with $H v_{0} = \lambda v_{0}$ and $E v_{0}=0$.

Now let us regard $\lambda$ as an abstract variable.
Let $\widehat{V_{\hbar,\lambda}}$ be a $\C[\lambda][[\hbar]]$-module freely generated by $\{\widehat{v_{0}},\widehat{v_{1}},\ldots,\}$, equipped with an $U_{\hbar}(\sltwo)$-module structure 
\begin{gather}
\label{eqn:verma}
\begin{cases}
 H \widehat{v}_{i} = (\lambda -2i) \widehat{v_{i}} \\
 E \widehat{v}_{i} = \widehat{v}_{i-1} \\
 F \widehat{v}_{i} = [i+1]_{q}[\lambda-i]_{q} \widehat{v}_{i+1}.
\end{cases}
\end{gather}
Here we put
\[ [\lambda-i]_{q} = \frac{\sinh(\half \hbar(\lambda-i))}{\sinh(\half \hbar)} = \frac{e^{\half \hbar(\lambda-i)} - e^{- \half \hbar(\lambda-i)}}{e^{\half \hbar}-e^{-\half \hbar}}. \]
We call $\widehat{V_{\hbar,\lambda}}$ a \emph{generic Verma module}. 

For $j=0,1,\ldots,$ define
\[ v_{j} = [\lambda]_{q}[\lambda-1]_{q}\cdots[\lambda-j+1]_{q} \widehat{v_{j}} \]
and let $V_{\hbar,\lambda}$ be the sub $U_{\hbar}(\sltwo)$-module of $\widehat{V_{\hbar,\lambda}}$ spanned by $\{v_{0},v_{1},\ldots \}$, with the action of $U_{\hbar}(\sltwo)$ given by
\begin{gather}
\label{eqn:vermamod}
\begin{cases}
 H v_{i} = (\lambda -2i) v_{i} \\
 E v_{i} = [\lambda+1-i]_{q} v_{i-1} \\
 F v_{i} = [i+1]_{q} v_{i+1}.
\end{cases}
\end{gather}

For $c \not \in \C^{*}-\{1,2,\ldots\}$, $\widehat{V_{\hbar,\lambda}}|_{\lambda = c}$ is isomorphic to $V_{\hbar,\lambda}|_{\lambda = c}$ because $[\lambda]_{q}[\lambda-1]_{q}\cdots[\lambda-j+1]_{q}$ is invertible for all $j$.
On the other hand, for $c \in\{1,2,\ldots,\}$, $v_{j}=0$ if $j>c$ and (\ref{eqn:vermamod}) shows that $V_{\hbar,\lambda}|_{\lambda = c}$ is nothing but the standard irreducible $U_{\hbar}(\sltwo)$-module of dimension $(c+1)$ whereas $\widehat{V_{\hbar,\lambda}}|_{\lambda = c}$ is infinite dimensional representation.

Let us define $\mathsf{R}: \widehat{V_{\hbar,\lambda}} \otimes \widehat{V_{\hbar,\lambda}} \rightarrow \widehat{V_{\hbar,\lambda}} \otimes \widehat{V_{\hbar,\lambda}}$ by $\mathsf{R} = e^{-\frac{1}{4}\hbar \lambda^{2}} T \mathcal{R}$, where $T:\widehat{V_{\hbar,\lambda}} \otimes \widehat{V_{\hbar,\lambda}} \rightarrow \widehat{V_{\hbar,\lambda}} \otimes  \widehat{V_{\hbar,\lambda}}$ is the transposition map $T(v\otimes w) = w \otimes v$, and $\mathcal{R}$ is the universal $R$-matrix (\ref{eqn:univR}).

By putting $z=q^{\lambda -1}=e^{\hbar(\lambda-1)}$, the action of $\mathsf{R}$ is given by the formula
\begin{gather}
\label{eqn:raction}
\begin{cases}
\mathsf{R}(\widehat{v}_{i} \otimes \widehat{v}_{j}) = z^{-\frac{i+j}{2}}q^{-\frac{i+j}{2}} \sum_{n=0}^{i} q^{(i-n)(j+n)} q^{\frac{n(n-1)}{4}} 
\frac{[n+j]_{q}!}{[n]_{q}![j]_{q}!}
\prod_{k=0}^{n-1}(z^{\half}q^{-\frac{1+k+j}{2}}-z^{-\half}q^{\frac{1+k+j}{2}}) \widehat{v}_{j+n}\otimes \widehat{v}_{i-n}. \\
\mathsf{R}(v_{i} \otimes v_{j}) = z^{-\frac{i+j}{2}}q^{-\frac{i+j}{2}} \sum_{n=0}^{i} q^{(i-n)(j+n)} q^{\frac{n(n-1)}{4}} \frac{[n+j]_{q}!}{[n]_{q}![j]_{q}!} \prod_{k=0}^{n-1}(z^{\half}q^{\frac{1-i+k}{2}}-z^{-\half}q^{-\frac{1-i+k}{2}}) v_{j+n}\otimes v_{i-n}
\end{cases}
\end{gather}

Let $\bL = \Z[q^{\pm 1},z^{\pm 1}] = \Z[e^{\pm \hbar},e^{\pm \hbar(\lambda-1)}] \subset \C[\lambda][[\hbar]]$ and let $V_{\bL}$ and $\widehat{V_{\bL}}$ be the sub free $\bL$-module of $\widehat{V_{\hbar,\lambda}}$ and $V_{\hbar,\lambda}$, spanned by $\{\widehat{v_{0}},\ldots\}$ and $\{v_{0},\ldots\}$, respectively.

Since all the coefficients of the action of $\mathsf{R}$ (\ref{eqn:raction}) lie in $\bL$, $\widehat{V_{\bL}}$ and $V_{\bL}$ are equipped with an $\bL B_{n}$-module structure. We denote the corresponding braid group representations by
\[ \widehat{\varphi_{\bL}}: B_{n} \rightarrow \GL( \widehat{V_{\bL}}{}^{\otimes n}), \ \ \  \varphi_{\bL}: B_{n} \rightarrow \GL( V_{\bL}^{\otimes n}). \] 

These are decomposed as finite dimensional representations as follows. For $m\geq 0$, define $\widehat{V_{n,m}} \subset \widehat{V_{\bL}}^{\otimes n}$ and $V_{n,m} \subset V_{\bL}^{\otimes n}$ by  

\begin{gather*}
\begin{cases}
\widehat{V_{n,m}}  = \textrm{ker}\; (q^{H}- q^{\frac{n\lambda -2m}{2}})  =  \textrm{span}\{\widehat{v}_{i_1}\otimes \cdots \otimes \widehat{v}_{i_n} \: | \: i_1+\cdots +i_n =m\}. \\
V_{n,m} = \textrm{ker}\; (q^{H}- q^{\frac{n\lambda -2m}{2}})  =  \textrm{span}\{v_{i_1}\otimes \cdots \otimes v_{i_n} \: | \: i_1+\cdots +i_n =m\} \end{cases}
\end{gather*}
By (\ref{eqn:raction}), the $B_{n}$-action preserves both $\widehat{V_{n,m}}$ and $V_{n,m}$ so we have linear representations 
\[ \widehat{\varphi^{V}_{n,m}}: B_{n} \rightarrow \GL(\widehat{V_{n,m}}) \ \textrm{ and } \varphi^{V}_{n,m}: B_{n} \rightarrow \GL(V_{n,m}) .\]
We call the $\bL B_{n}$-module $\widehat{V_{n,m}}$  the \emph{(generic) weight space} of weight $q^{\frac{n\lambda -2m}{2}}$.

By definition, as $\bL B_{n}$-modules, $\widehat{V_{\bL}}^{\otimes n}$ and $V_{\bL}^{\otimes n}$ split as
\begin{gather}
\label{eqn:Vnm}
\begin{cases}
\widehat{V_{\bL}}^{\otimes n} \cong \bigoplus_{m=0}^{\infty} \widehat{V_{n,m}} \\
V_{\bL}^{\otimes n} \cong \bigoplus_{m=0}^{\infty} V_{n,m}
\end{cases}
\end{gather}

Finally, we define the {\em space of (generic) null vectors} $\widehat{W_{n,m}}$ by
\[  \widehat{W_{n,m}} = \textrm{Ker}\, (E) \cap \widehat{V_{n,m}}. \]
Since the action of $B_{n}$ commutes with the action of $U_{q}(\sltwo)$, we have linear representation 
\[ \widehat{\varphi^{W}_{n,m}}: B_{n} \rightarrow \GL(\widehat{W_{n,m}}).\]

In \cite[Lemma 13]{jk}, it is shown that for $k=1,\ldots,m$, the map $F^{m-k} : \widehat{W_{n,k}} \rightarrow \widehat{V_{n,m}}$
is injection and that over the quotient field, $\widehat{V_{n,m}}$ splits as
\begin{equation}
\label{eqn:splitVnm}
\widehat{V_{n,m}} \cong_{\mQ} \bigoplus_{k=0}^{m} F^{m-k} \widehat{W_{n,k}} \cong_{\mQ} \bigoplus_{k=0}^{m} \widehat{W_{n,k}}, 
\end{equation}
hence combining with (\ref{eqn:Vnm}), we conclude that the $\bL B_{n}$-module $\widehat{V_{\bL}}^{\otimes n}$ splits, over the quotient field,
\begin{equation}
\label{eqn:splitV}
\widehat{V_{\bL}}^{\otimes n} \cong_{\mQ} \bigoplus_{m=0}^{\infty} \bigoplus_{k=0}^{m} \widehat{W_{n,k}}. 
\end{equation}

\subsection{Lawrence's homological representations}

Here we briefly review the definition of (geometric) Lawrence's representation $L_{n,m}$. An explicit matrix of $L_{n,m}(\sigma_{i})$ and some details will be given in Appendix.

For $i= 1,2,\ldots,n$, let $p_{i}=i \in \C$ and $D_{n}=\{ z \in \C \: | \: |z| \leq n+1 \} - \{p_{1},\ldots,p_{n}\}$ be the $n$-punctured disc. We identify the braid group $B_{n}$ with the mapping class group of $D_{n}$ so that the standard generator $\sigma_{i}$ corresponds to the {\em right-handed} half Dehn twist that interchanges the $i$-th and $(i+1)$-st punctures.

For $m>0$, let $C_{n,m}$ be the unordered configuration space of $m$-points in $D_{n}$,
\[ C_{n,m} =\{ (z_{1},\ldots,z_{m}) \in D_{n} \: | \: z_{i} \neq z_{j} \;(i\neq j) \} \slash S_{m} \]
where $S_{m}$ is the symmetric group acting as permutations of the indices.
For $i=1,\ldots,n$, let $d_{i}=(n+1)e^{(\frac{3}{2}+ i\varepsilon)\pi \sqrt{-1}}  \in \partial D_{n}$ where $\varepsilon >0$ is sufficiently small number, and we take $\mathbf{d} = \{d_{1},\ldots,d_{m}\}$ as a base point of $C_{n,m}$.

The first homology group $H_{1}(C_{n,m}; \Z)$ is isomorphic to $\Z^{\oplus n} \oplus \Z$, where the first $n$ components correspond to the meridians of the hyperplanes $\{z_{1}=p_{i}\}$ $(i=1,\ldots,n)$ and the last component corresponds to the meridian of the discriminant $\bigcup_{1\leq i < j \leq n }\{z_{i}=z_{j}\}$.

Let $\alpha: \pi_{1}(C_{n,m}) \rightarrow \Z^{2} = \langle x,d\rangle$ be the homomorphism obtained by composing the Hurewicz homomorphism $\pi_{1}(C_{n,m}) \rightarrow H_{1}(C_{n,m}; \Z)$ and the projection 
\[ H_{1}(C_{n,m}; \Z) = \Z^{\oplus n} \oplus \Z = \langle x_1,\ldots, x_n \rangle \oplus \langle d \rangle \rightarrow \langle x_1+ \cdots + x_n \rangle \oplus \langle d \rangle = \langle x \rangle \oplus \langle d \rangle =  \Z \oplus \Z.  \]

Let $\pi: \tC_{n,m} \rightarrow C_{n,m}$ be the covering corresponding to $\textrm{Ker}\,\alpha$. We fix a lift $\widetilde{\mathbf{d}} \in \pi^{-1}(\mathbf{d}) \subset \tC_{n,m}$ and use $\widetilde{\mathbf{d}}$ as a base point of $\tC_{n,m}$. By identifying $x$ and $d$ as deck translations, 
 $H_{m}(\tC_{n,m};\Z)$ has a structure of $\Z[x^{\pm 1},d^{\pm 1}]$-module.
Actually, it is known that $H_{m}(\tC_{n,m};\Z)$ is a free $\Z[x^{\pm 1},d^{\pm 1}]$-module of rank $\binom{m+n-2}{m}$.

We will actually use $H_{m}^{lf}(\tC_{n,m};\Z)$, the homology of locally finite chains, and consider a free sub $\Z[x^{\pm 1},d^{\pm 1}]$-module $\mathcal{H}_{n,m} \subset H_{m}^{lf}(\tC_{n,m};\Z)$ of rank $\binom{m+n-2}{m}$, spanned by homology classes represented by certain geometric objects called \emph{mutliforks}. The subspace $\mathcal{H}_{n,m}$ is preserved by $B_{n}$ actions, hence by using a natural basis of $\mathcal{H}_{n,m}$ called \emph{standard mutliforks}, we get a linear representation 
\[ L_{n,m} \co B_{n} \rightarrow \GL(\mathcal{H}_{n,m}) = \GL( {\textstyle \binom{m+n-2}{m}}; \Z[x^{\pm 1}, d^{\pm 1}]). \]
which we call {\em (Geometric) Lawrence's representation}.

In the case $m=1$ the discriminant $\bigcup_{1\leq i < j \leq n }\{z_{i}=z_{j}\}$ is empty so the variable $d$ does not appear. The representation 
\[ L_{n,1} \co  B_{n} \rightarrow \GL( n-1; \Z[x^{\pm1}]). \]
coincides with the reduced Burau representation.
The representation $L_{n,2}$ is often called the \emph{Lawrence-Krammer-Bigelow representation}, which is extensively studied in \cite{big0,kra,kra2} and known to be faithful.

\begin{remark}
In general, the braid group representations $\mathcal{H}_{n,m}$, $H_{m}(\tC_{n,m};\Z)$ and $H_{m}^{lf}(\tC_{n,m};\Z)$ are not isomorphic each other. However, there is an open dense subset $U \subset \C^{2}$ such that if we specialize $x$ and $d$ to complex parameters in $U$, then these three representations are isomorphic \cite{koh2}. Namely, all representations are \emph{generically} identical. 
In particular, they are all isomorphic over the quotient field, $\mathcal{H}_{n,m} \cong_{\mQ} H_{m}(\tC_{n,m};\Z) \cong_{\mQ} H_{m}^{lf}(\tC_{n,m};\Z)$.
\end{remark}

The following  well-known result will explain why the MMR conjecture is true.

\begin{proposition}
\label{prop:reduced}
When we specialize $d=-1$, then the Lawrence's representation $L_{n,m}$ is equal to $\Sym^{m}L_{n,1}$, the $m$-th symmetric power of the reduced Burau representation $L_{n,1}$.
\end{proposition} 

Proposition \ref{prop:reduced} is directly seen by the explicit matrix formula of $L_{n,m}$ (\ref{eqn:formula}) in Appendix.

Roughly saying, when we specialized $d=-1$, that is, when we ignore the effect of discriminant $\bigcup_{1\leq i < j \leq n }\{z_{i}=z_{j}\}$, we forget an interaction of points. Then a natural inclusion $C_{n,m} \rightarrow C_{n,1}^{m} \slash S_m$ induces an isomorphism 
\[ H_{m}(\widetilde{C_{n,m}};\Z)|_{d=-1} \rightarrow H_{m}(\widetilde{C_{n,1}}^{m} \slash S_m;\Z) \cong H_{1}(\widetilde{C_{n,1}};\Z)^{\otimes m} \slash S_m = \Sym^{m}H_{1}(\widetilde{C_{n,1}};\Z) \]
of the braid group representations.

Here, we remark that somewhat confusing minus sign of $d$ comes from the convention of the orientation of submanifold representing an element of $H_{m}(\widetilde{C_{n,m}};\Z)$, as we will explain in Appendix.

\subsection{Identification and specializations of quantum and homological representations}

Here we summarize relations of braid group representations introduced in previous sections.
First, generically quantum representation is identified with Lawrence's representation.

\begin{theorem}\cite[Corollary 4.6]{i},\cite{koh2}
\label{theorem:WisH}
As a braid group representation, there is an isomorphism
\[ \widehat{W_{n,m}} \cong_{\mQ} \mathcal{H}_{n,m}|_{x=z^{-1}q,d=-q}. \]
\end{theorem}

For $\alpha \in \{2,3,\ldots\}$, let $V_{\alpha}$ be the $\alpha$-dimensional irreducible $U_{q}(\sltwo)$ module and $\varphi_{\alpha}: B_{n} \rightarrow \GL(V_{\alpha}^{\otimes n})$ be the quantum representation.
Let $e:B_{n} \rightarrow \Z$ be the exponent sum map given by $e(\sigma_{i}^{\pm}) = \pm 1$. 
The representation $\varphi_{\alpha}$ can be recovered from a version of generic quantum representation as follows.

\begin{proposition}
\label{prop:LandQ}
Let $\beta \in B_{n}$ and $\alpha \in \{2,3,\ldots\}$.
Then
\[ e^{\frac{1}{4}\hbar(\alpha-1)^{2}e(\beta)}\varphi_{\bL}(\beta)|_{\lambda = \alpha -1} = \varphi_{\alpha}(\beta) \]
\end{proposition}
\begin{proof}
As we have seen, as $U_{q}(\sltwo)$-module we have an isomorphism $V_{\alpha} \cong V_{\bL}|_{\lambda=\alpha -1}$. The formula follows from this observation and the definition $\mathsf{R} = e^{-\frac{1}{4}\hbar \lambda^{2}} T \mathcal{R}$.
\end{proof}

Next we observe when we specialize $\lambda$ as integer, certain symmetry appears.

\begin{lemma}
\label{lemma:symmetry}
$V_{n,m}|_{\lambda = \alpha -1} \cong  V_{n,n\alpha-n-m}|_{\lambda = \alpha -1}$.
\end{lemma}
\begin{proof}
Let us put 
\[ \mathsf{R}(v_{i}\otimes v_{j}) = \sum_{n=0}^{\infty} a_n v_{j+n} \otimes v_{i-n}, \ \ \mathsf{R}(v_{\lambda-j}\otimes v_{\lambda-i}) = \sum_{n=0}^{\infty} b_n v_{\lambda-i+n} \otimes v_{\lambda-j-n}, \]
where $a_n,b_n \in \bL|_{z=q^{\alpha-1}}\cong \Z[q^{\pm 1}]$.
We show $a_{n}=b_{n}$ for all $i,j$. This shows an equivalence of $R$-operators hence proves the desired isomorphism.

Note that when the weight variable $\lambda$ is specialized as a positive integer $\alpha-1$, $v_{k}=0$ whenever $k \geq \lambda$, $a_n=b_n=0$ if $n > \min\{i,\lambda-j\}$. Hence we consider the case $n \leq \min\{i,\lambda-j\}$.

By (\ref{eqn:raction}), with putting $z=q^{\lambda-1}$, we have
\begin{eqnarray*}
 b_n & = &q^{\frac{\lambda}{2}(2\lambda-i-j)}q^{(\lambda-j-n)(\lambda-i+n)}q^{\frac{n(n-1)}{4}}\frac{[n+\lambda-i]_{q}!}{[n]_{q}![\lambda-i]_{q}!} \prod_{k=0}^{n-1}(q^{\half(j+k+1)}-q^{-\half(j+k+1)})\\
 & = & q^{-\frac{\lambda}{2}(i+j)}q^{(i-n)(j+n)}q^{\frac{n(n-1)}{4}}\frac{[n+\lambda-i]_{q}!}{[n]_{q}![\lambda-i]_{q}!} \prod_{k=0}^{n-1}(q^{\half(j+k+1)}-q^{-\half(j+k+1)})
\end{eqnarray*}

Since 
\[ \prod_{k=0}^{n-1}(q^{\half(j+k+1)}-q^{-\half(j+k+1)}) = \frac{[j+1]_{q}\cdots [j+n]_{q}}{(q^{\half}-q^{-\half})^{n}} = \frac{[j+n]_{q}!}{[j]_{q}!(q^{\half}-q^{-\half})^{n}}\]
we conclude
\begin{eqnarray*}
 b_n & = & q^{-\frac{\lambda}{2}(i+j)}q^{(i-n)(j+n)}q^{\frac{n(n-1)}{4}}\frac{[n+\lambda-i]_{q}!}{[n]_{q}![\lambda-i]_{q}!} \cdot \frac{[j+n]_{q}!}{[j]_{q}!(q^{\half}-q^{-\half})^{n}}\\
& = & q^{-\frac{\lambda}{2}(i+j)}q^{(i-n)(j+n)}q^{\frac{n(n-1)}{4}}\frac{[j+n]_{q}!}{[n]_{q}![j]_{q}!} \cdot \frac{[n+\lambda-i]_{q}!}{[\lambda-i]_{q}!(q^{\half}-q^{-\half})^{n}}\\
& = &  q^{-\frac{\lambda}{2}(i+j)}q^{(i-n)(j+n)}q^{\frac{n(n-1)}{4}}\frac{[j+n]_{q}!}{[n]_{q}![j]_{q}!} \prod_{k=0}^{n-1}(q^{\half(\lambda-i+k+1)}-q^{-\half(\lambda-i+k+1)})\\
& = & a_n
\end{eqnarray*}
\end{proof}

\section{A topological formula for the loop expansion of the Colored Jones poynomials}

Now we are ready to prove our main result, a topological formula of the loop expansion of the colored Jones polynomials.

\begin{theorem}
\label{theorem:main}
Let $K$ be an oriented knot in $S^{3}$ represented as a closure of an $n$-braid $\beta$. Then the loop expansion of the colored Jones polynomial is given by
\[ CJ_{K}(\hbar,z) = \frac{ z^{-\half e(\beta)} q^{\half e(\beta)} }{z^{\half} -z^{-\half} } \sum_{m=0}^{\infty}
(z^{\frac{n}{2}} q^{\frac{1-n-2m}{2}} - z^{-\frac{n}{2}} q^{\frac{-1+n+2m}{2}} ) tr L_{n,m}(\beta)|_{x=qz^{-1},d=-q}.
\]
Here we put $q=e^{\hbar}$.
\end{theorem}
\begin{proof}
The colored Jones polynomial $J_{K,\alpha}(q)$ is defined by
\[ J_{K,\alpha}(q)= \frac{1}{[\alpha]_{q}}q^{-\frac{1}{4}(\alpha^{2}-1)e(\beta)} \tr(q^{\frac{H}{2}} \varphi_{\alpha}(\beta)). \]
By Proposition \ref{prop:LandQ},
\[ J_{K,\alpha}(q)= \frac{1}{[\alpha]_{q}}q^{-\frac{1}{4}(\alpha^{2}-1)e(\beta)} q^{\frac{1}{4}(\alpha-1)^{2}e(\beta)} \tr(q^{\frac{H}{2}}\varphi_{\bL}(\beta)|_{\lambda = \alpha -1}). \]
The colored Jones function $CJ_{K}(z,\hbar)$ is obtained by taking the limit $\alpha \to \infty$ keeping $z = e^{\hbar \alpha}$ is constant, namely treating $\hbar \alpha$ as an independent variable:
\[ CJ_{K}(\hbar, z) = \frac{q^{\half}-q^{-\half}}{z^{\half} -z^{-\half} }
z^{-\half e(\beta)}q^{\half e(\beta)} 
\lim_{
\begin{subarray}{c}
\alpha \to \infty \\ \hbar \alpha: \textrm{constant}
\end{subarray}
}  \tr(q^{\frac{H}{2}} \varphi_{\bL}(\beta)|_{\lambda = \alpha -1}) \]

We compute the limit as follows (see Figure \ref{fig:summary} for a diagrammatic summary of the computation.)

\begin{figure}[htbp]
\centerline{\includegraphics[bb=138 444 557 708, width=140mm]{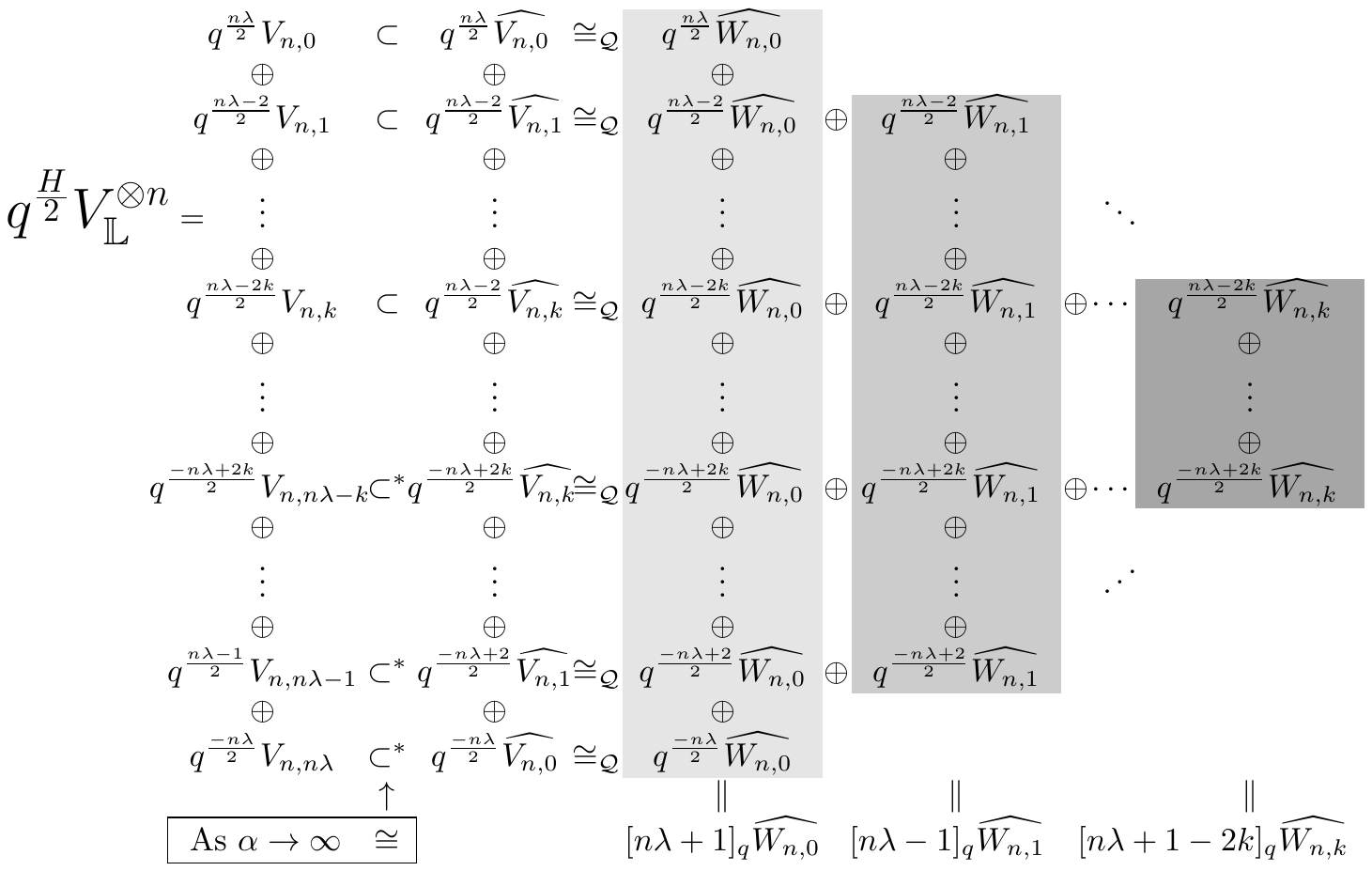}}

\caption{ This diagram explains how to compute the desired limit.
In the diagram, every representations are understood as taking specialization $\lambda = \alpha -1$. The notation $\subset^{*}$ means that we regard $V_{n,n\lambda-i}$ as sub module of $\widehat{V_{n,i}}$, by using isomorphism $V_{n,n\lambda-i}\cong V_{n,i}$ in Lemma \ref{lemma:symmetry}.}
\label{fig:summary}
\end{figure}

Since $V_{n,i}|_{\lambda = \alpha -1} = 0$ if $i >n\lambda = n(\alpha -1)$, by (\ref{eqn:splitV}), as a $\bL B_{n}$-module $V_{\bL}^{\otimes n} \cong \bigoplus_{i=0}^{n\lambda} V_{n,i}$. Moreover, $q^{\frac{H}{2}}$ acts on $V_{n,i}$ as a scalar multiple by $q^{\frac{n\lambda -2i}{2}}$, so 
\[ \tr(q^{\frac{H}{2}} \varphi_{\bL}(\beta)|_{\lambda = \alpha -1}) = \sum_{i=0}^{n\lambda} q^{\frac{n\lambda -2i}{2}} \tr (\varphi^{V}_{n,i}(\beta)|_{\lambda = \alpha-1}). \]
By Lemma \ref{lemma:symmetry}, we identify the braid group representation $V_{n,n\lambda-k}$ as $V_{n,k}$ for $k \leq \frac{n\lambda}{2}$, and then regard each $V_{n,i}$ as a sub $\bL B_{n}$-module of $\widehat{V_{n,i}}$. Recall that $\tr (\varphi^{V}_{n,i}(\beta)|_{\lambda = \alpha-1})$ is equal to $\tr (\widehat{\varphi^{V}_{n,i}}(\beta)|_{\lambda = \alpha-1})$ when $\alpha$ is treated as independent variable. By (\ref{eqn:splitVnm}), over quotient field, $\widehat{V_{n,i}}$ splits as $\bigoplus_{m=0}^{i}\widehat{W_{n,m}}$, hence
\[  \tr (\widehat{\varphi^{V}_{n,i}}(\beta)|_{\lambda = \alpha-1}) = \sum_{m=0}^{\min\{m, n\lambda -m\}} tr (\widehat{\varphi^{W}_{n,m}}(\beta)|_{\lambda = \alpha-1}) \]
This shows
\begin{eqnarray*}
\sum_{i=0}^{\alpha-1} q^{\frac{n\lambda -2i}{2}} \tr (\widehat{\varphi^{V}_{n,i}}(\beta)|_{\lambda = \alpha-1}) & =  & \sum_{i=0}^{\alpha-1} q^{\frac{n\lambda -2i}{2}} \sum_{m=0}^{\min\{m, n\lambda -m\}} \tr (\widehat{\varphi^{W}_{n,i}}(\beta)|_{\lambda = \alpha-1}) \\
& = & \sum_{m=0}^{\frac{n\lambda}{2}} \sum_{i=m}^{n\lambda-m} q^{\frac{n\lambda -2i}{2}} \tr (\widehat{\varphi^{W}_{n,m}}(\beta)|_{\lambda = \alpha-1})\\
& = &  \sum_{m=0}^{\frac{n\lambda}{2}} [n\lambda +1-2m]_{q} \tr (\widehat{\varphi^{W}_{n,m}}(\beta)|_{\lambda = \alpha-1})
\end{eqnarray*}

By Theorem \ref{theorem:WisH}, when we treat $\alpha$ as independent variable, 
\[ \tr (\widehat{\varphi^{W}_{n,i}}(\beta)) = \tr(L_{n,m}(\beta)|_{x=z^{-1}q, d=-q})\]
hence 
\begin{eqnarray*}
\lim_{
\begin{subarray}{c}
\alpha \to \infty \\ \hbar \alpha: \textrm{constant}
\end{subarray}
}\!\!\!\!
\tr(q^{\frac{H}{2}}\circ \varphi_{\bL}(\beta)|_{\lambda = \alpha -1})
& = &
\lim_{
\begin{subarray}{c}
\alpha \to \infty \\ \hbar \alpha: \textrm{constant}
\end{subarray}
}
\!\!\!\! \sum_{m=0}^{ \min\{m, n\lambda -m\}} [n\lambda+1-2m]_{q} \tr (L_{n,m}(\beta)|_{x=z^{-1}q, d=-q})\\
& = & \sum_{m=0}^{\infty} \frac{ z^{\frac{n}{2}} q^{\half(1-n-2m)} - z^{-\frac{n}{2}} q^{-\half(1-n-2m)} }{q^{\half}-q^{-\half}} \tr (L_{n,m}(\beta)|_{x=z^{-1}q, d=-q}).\\
\end{eqnarray*}
Therefore we conclude
\[ CJ_{K}(\hbar,z) = \frac{ z^{-\half e(\beta)} q^{\half e(\beta)} }{z^{\half} -z^{-\half} } \sum_{m=0}^{\infty}(z^{\frac{n}{2}} q^{\half(1-n-2m)} - z^{-\frac{n}{2}} q^{-\half(1-n-2m)}) \tr (L_{n,m}(\beta)|_{x=z^{-1}q, d=-q}).
\]
\end{proof}

As we have mentioned, Theorem \ref{theorem:main} provides an alternative, direct method to compute the loop expansion of the colored Jones polynomial although actual computation may be quite hard, since one should know $L_{n,m}(\beta)$ for all $m$. Here we give sample calculations.

\begin{example}[Unknot]
Let us consider the unknot $K$ represented as a closure of 2-braid $\sigma_{1}$. The trace of Lawrence's representation is given by 
$\tr L_{2,m}(\sigma_{1})= (-x)^{m}(-d)^{\binom{m}{2}}$ so 
\begin{eqnarray*}
CJ_{\textsf{Unknot}}(z,\hbar) & = & \frac{ z^{-\half} q^{\half} }{z^{\half} -z^{-\half} } \sum_{m=0}^{\infty}
(z q^{\frac{-1-2m}{2}} - z^{-1} q^{\frac{1+2m}{2}} ) (-z)^{-m}q^{\binom{m+1}{2}}\\
& = & \frac{ z^{-\half} q^{\half} }{z^{\half} -z^{-\half} }(zq^{-\half} -q^{-\half}) \\
& = & 1.
\end{eqnarray*}
\end{example}

\begin{example}[$(2,p)$-torus knot]

More generally, let us consider $(2,p)$-torus knot $T(2,p)$ represented as a closure of 2-braid $\sigma_{1}^{p}$.
The trace of Lawrence's representation is given by $\tr L_{2,m}(\sigma_{1}^{p})= (-x^{p})^{m}(-d^{p})^{\binom{m}{2}}$ so 
\begin{equation}
\label{eqn:CJ}
CJ_{T(2,p)}(z,\hbar)  =  \frac{ z^{-\half p} q^{\half p} }{z^{\half} -z^{-\half} } \sum_{m=0}^{\infty}
(z q^{\frac{-1-2m}{2}} - z^{-1} q^{\frac{1+2m}{2}} ) (-z^{-p})^{m}q^{\binom{m+1}{2}p}
\end{equation}

To compute the $1$-loop part, let us put $\hbar=0$. Then
\begin{eqnarray*}
V^{(0)}_{T(2,p)} (z)  & = & \frac{ z^{-\half p} }{z^{\half} -z^{-\half} } \sum_{m=0}^{\infty}
(z - z^{-1}  ) (-z^{-p})^{m} =  
(z^{\half}+z^{-\half})z^{-\frac{p}{2}}\sum_{m=0}^{\infty} (-z^{-p})^{m}\\
&= &(z^{\half} + z^{-\half})z^{-\frac{p}{2}} \frac{1}{1 + z^{-p}}\\
& = & \frac{z^{\half}+z^{-\half}}{z^{\half p} + z^{-\half p}}
\end{eqnarray*}
which is equal the inverse of the Alexander-Conway polynomial of $T(2,p)$.

Next let us compute the 2-loop part. By putting $q = e^{\hbar}$ and looking at the coefficient of $\hbar$ in (\ref{eqn:CJ}), we have
\begin{eqnarray*}
 V^{(1)}_{T(2,p)}(z) & = & \frac{z^{-\half p}}{z^{\half} -z^{-\half} }\left(
\frac{z}{2}\sum_{m=0}^{\infty} [pm^{2}+(p-2)m+(p-1)](-z^{-pm})\right. \\
& & \ \ \ \ \ \ \ \ \ \ \ \ \ \ \ \left. - \frac{z^{-1}}{2} \sum_{m=0}^{\infty} [pm^{2}+(p+2)m+(p+1)](-z^{-pm}) \right)
\end{eqnarray*}
In the ring of formal power series $\C[[z,z^{\pm 1}]]$, we have
\[ \sum_{m=0}^{\infty} z^{m} = \frac{1}{1-z}, \ \sum_{m=0}^{\infty} m z^{m}=\frac{z}{(1-z)^{2}},\ \sum_{m=0}^{\infty} m^{2}z^{m} = \frac{z+z^{2}}{(1-z)^{3}}.\]
Hence 
\[ V^{(1)}_{T(2,p)}(z) = \frac{1}{(z^{\half}-z^{-\half})(z^{\half p} + z^{-\half p})^{3}} \left( (p-1)z^{p+1} -(p+1)z^{p-1} + (p+1)z^{-p+1} -(p-1)z^{-p-1}\right).\]
For example, 
\[ V^{(1)}_{T(2,3)}(z) = \frac{(z^{4}-2z^{2}+2-2z^{-2}+z^{-4})}{\Delta_{T(2,3)}(z)^{3}} = \frac{(z^{4}-2z^{2}+2-2z^{-2}+z^{-4})}{(z-1+z^{-1})^{3}}.\]
\end{example}

Now it is a direct consequence that the 1-loop part is the inverse of the Alexander-Conway polynomial.

\begin{corollary}[Melvin-Morton-Ronzansky conjecture]
\label{cor:MMR}
\[ V^{(0)}(z) = \frac{1}{\Delta_{K}(z)}\]
\end{corollary}
\begin{proof}
The $1$-loop part $V^{(0)}(z)$ is obtained by putting $\hbar=0$ in the formula of Theorem \ref{theorem:main}.
Since $d=-q=-e^{\hbar}$, in a homological representation, putting $\hbar=0$ corresponds to putting $d=-1$. 
As we have pointed out in Proposition \ref{prop:reduced},
\[ L_{n,m}(\beta)|_{d=-1} = \Sym^{m}L_{n,1}(\beta). \]

Therefore, by Theorem \ref{theorem:main}, the $1$-loop part $V^{(0)}(z)$ is written as
\begin{eqnarray*}
V^{(0)}(z) & = &  z^{-\half e(\beta)} \frac{z^{\frac{n}{2}}-z^{-\frac{n}{2}}}{z^{\half}-z^{-\half}} \sum_{i=0}^{\infty}\tr L_{n,i}(\beta)|_{x=z^{-1}, d=-1} \\
& = & z^{-\half e(\beta)} \frac{z^{\frac{n}{2}}-z^{-\frac{n}{2}}}{z^{\half}-z^{-\half}} \sum_{i=0}^{\infty} \tr (\Sym^{i}L_{n,1})(\beta)|_{x=z^{-1}}.
\end{eqnarray*}
MacMahon Master Theorem says that
\[ \sum_{i=0}^{\infty} \tr (\Sym^{i}L_{n,1})(\beta) = \det(I-L_{n,1}(\beta))^{-1} \]
hence
\begin{eqnarray*}
V^{(0)}(z) = z^{-\half e(\beta)} \frac{z^{\frac{n}{2}}-z^{-\frac{n}{2}}}{z^{\half}-z^{-\half}} \frac{1}{\det(I-L_{n,1}(\beta)|_{x=z^{-1}})} = \frac{1}{\Delta_{K}(z^{-1})} = \frac{1}{\Delta_{K}(z)}. \end{eqnarray*}
\end{proof}

\section{Entropy and colored Jones polynomials}
\label{sec:ent}

In this section we give an application of topological interpretation of quantum representations. 

\subsection{Entropy estimates from configuration space}

For a homeomorphism of a compact topological space or a metric space $f: X \rightarrow X$, there is a fundamental numerical invariant $h(f) \in \R$ of topological dynamics called the \emph{(topological) entropy}. 

Let $\overline{C}_{m}(X)$ and $C_{m}(X)$ be the \emph{ordered} and \emph{unordered configuration space} of $m$-points of $X$,
\[\overline{C}_{m}(X) =\{(x_{1}\ldots,x_{m}) \in X^{m} \: | \: x_{i}\neq x_{j}\}, \ \ C_{m}(X) = \overline{C}_{m}(X) \slash S_{m}\]
where $S_{m}$ is the symmetric group that acts as permutations of the coordinates.
Then $f$ induces the continuous maps $\overline{C}_{m}(f): \overline{C}_{m}(X) \rightarrow \overline{C}_{m}(X)$ and $C_{m}(f): C_{m}(X) \rightarrow C_{m}(X)$, respectively.

Note that $\overline{C}_{m}(X) \subset X^{m}$ is invariant under $f^{\times m}:X^{m} \rightarrow X^{m}$ so
\[ h(\overline{C}_m(f)) \leq h(f^{\times m}) = m h(f). \]
The unordered configuration space $\overline{C}_{m}(X)$ is a finite cover of $C_{m}(X)$ so $ h(C_m(f)) = h(\overline{C}_{m}(f))$.

Now for $A \in \GL(n;\C)$ let $\rho(A)$ be its spectral radius of $A$, the maximum of the absolute value of the eigenvalues of $A$.
It is known that if $C_{m}(f)$ is nice enough (see \cite{fr}, for sufficient conditions for inequality (\ref{eqn:Shubconj}) to hold), then spectral radius of the induced action on homology provides a lower bound of the entropy
\begin{equation}
\label{eqn:Shubconj}
\log \rho(C_{m}(f)_{*}: H_{*}(C_{m}(f),\Z) \rightarrow H_{*}(C_{m}(f),\Z)) \leq h(C_m(f)).
\end{equation}
Hence if $f$ is good enough, by using configuration spaces we have an estimate of entropy
\begin{equation}
\label{eqn:ent}
\log \rho(C_{m}(f)_{*}) \leq m h(f). 
\end{equation}

The above considerations nicely fit for the braid groups. Let us regard the braid group $B_n$ as the mapping class group of $n$-punctured disc $D_{n}$. The entropy of braid $\beta \in B_{n}$ is defined by the infimum of entropy of homeomorphisms representing $\beta$,
\[ h(\beta) = \inf\{h(f)\: | f: D_n \rightarrow D_n, [f]=\beta \in MCG(D_{n})=B_n \}.\]
By Nielsen-Thurston classification \cite{flp,th}, there is a representative homeomorphism $f_{\beta}$ that attains the infimum so $h(\beta)=h(f_{\beta})$. In particular, if $\beta$ is pseudo-Anosov, then a psuedo-Anosov representative attains the infimum. By abuse of notation, we will use the same symbol $\beta$ to mean its representative homeomoprhism $f_{\beta}$ that attains the infimum of the entropy. 

As Koberda shows in \cite{kob}, the inequality (\ref{eqn:Shubconj}) holds in the case $X$ is surface. This implies that Lawrence's representation gives an estimate of entropy.

\begin{theorem}
\label{theorem:estimate}
For an $n$-braid $\beta$,
\[ \sup_{|x|=1,|d|=1} \log \rho(L_{n,m}(\beta)) \leq m h(\beta)\]
\end{theorem}
\begin{proof}
Let $\widetilde{C}$ be a finite covering of the unordered configuration space $C_{n,m} = C_{m}(D_{n})$. If the action of $\beta$ on $C_{n,m}$ lifts, then by (\ref{eqn:ent}), $\log \rho(\widetilde{\beta_{\widetilde{C}}}{}_{*} ) \leq mh(\beta)$
holds, where $\widetilde{\beta_{\widetilde{C}}}: \widetilde{C} \rightarrow \widetilde{C}$ denotes the lift of a homeomorphism $\beta$.

For non-negative integers $A,B$,
Let $\widetilde{C}=\widetilde{C}_{A,B}$ be a finite abelian covering of $C_{n,m}$ that corresponds to the kernel of $\alpha_{A,B}:\pi_{1}(C_{n,m}) \rightarrow \Z\slash A \Z \oplus \Z \slash B\Z$,
where $\alpha_{A,B}$ is given by the compositions
\[ 
\xymatrix{
\pi_{1}(C_{n,m}) \ar[r]^(0.4){\alpha} & \Z\oplus \Z \cong \langle x \rangle \oplus \langle d \rangle \ar[r]& (\langle x \rangle \slash x^{A}) \oplus (\langle d \rangle \slash d^{B})= \Z\slash A \Z \oplus \Z \slash B\Z.
} 
\]
The standard topological argument, using the eigenspace decompositions for the deck translations (See \cite{bb} for the case $\Z$-covering, the case of the reduced Burau representation $L_{n,1}$. The same argument applies to the case $\Z^{2}$-covering) shows that for $a=1,\ldots,A-1$ and $b=1,\ldots, B-1$,
\[ \rho( L_{n,m}(\beta)|_{ x= e^{\frac{2\pi a \sqrt{-1}}{A}}, d=e^{\frac{2\pi b \sqrt{-1}}{B} }})  \leq \rho( \widetilde{\beta}_{\widetilde{C}_{A,B}}). \]
The sets $\{ (e^{ \frac{2 \pi a \sqrt{-1}}{A}},e^{\frac{2 \pi b\sqrt{-1}}{B}}) \in \: | \:  A,a,B,b \in \Z \}$ is dense in $S^{1} \times S^{1}=\{(x,d) \in \C^{2}\: | \: |x|=|d|=1\}$, hence we get desired inequality.
\end{proof}

Since generically one can identify the quantum representation with Lawrence's representation, quantum representations also provide estimates of entropy. 

\begin{theorem}
\label{theorem:qestimate}
Let $\beta$ be an $n$-braid.
\begin{enumerate}
\item $\sup_{|q|=1,|z|=1} \log \rho(\widehat{\varphi_{n,m}^{W}}(\beta)) \leq m  h(\beta)$.
\item $\sup_{|q|=1,|z|=1} \log \rho(\widehat{\varphi_{n,m}^{V}}(\beta)) \leq m h(\beta)$.
\item $\sup_{|q|=1} \log \rho(\varphi_{\alpha}(\beta)) \leq \frac{n\alpha-n}{2}h(\beta)$.
\end{enumerate}
\end{theorem}

\begin{proof}
The assertions (1) and (2) follows from Theorem \ref{theorem:estimate} and Theorem \ref{theorem:WisH}.
To see (3), recall that as a $\bL B_{n}|_{z=q^{\alpha-1}} = \C[q^{\pm 1}] B_{n}$-module, we have
\[ V_{\alpha}^{\otimes n} \cong V_{\bL}|_{z=q^{\alpha-1}} \subset \bigoplus_{m=0}^{n(\alpha-1)} V_{n,m}|_{\lambda = \alpha-1}. \]
Moreover, by Lemma \ref{lemma:symmetry} $V_{n,m}|_{\lambda = \alpha-1} \cong V_{n,n(\alpha-1)-m}|_{\lambda = \alpha-1}$. 
Therefore,
\[ \sup_{|q|=1} \rho (\varphi_{\alpha}(\beta)) \leq 
\max_{1 \leq m \leq \frac{n(\alpha-1)}{2}} \sup_{|q|=1} (\varphi^{V}_{n,m}(\beta)|_{z=q^{\alpha-1}}) \leq 
\max_{1 \leq m \leq \frac{n(\alpha-1)}{2}} \sup_{|q|=1} \rho(\widehat{\varphi^{V}_{n,m}}(\beta)|_{ z=q^{\alpha-1}}).\]
By (1), we conclude
\[ \sup_{|q|=1} \log \rho(\varphi_{\alpha}(\beta)) \leq \frac{n(\alpha-1)}{2}h(\beta). \]
\end{proof}

\subsection{Quantum $\sltwo$ invariants and entropy}

An estimates in Theorem \ref{theorem:qestimate} suggests a new relationship between quantum invariants and entropy of braids.

For $\alpha \in \{2,3\ldots,\}$, let $Q^{\sltwo;V_{\alpha}}_{K}(q)= \tr(q^{\frac{H}{2}}\varphi_{\alpha}(\beta)) = [\alpha]_{q}J_{\alpha,K}(q)$ be the quantum $(\sltwo,V_{\alpha})$-invariant of the knot $K$, another common normalization of the colored Jones polynomials used to define quantum invariants of 3-manifolds.

\begin{theorem}
\label{theorem:JvsE}
Let $K$ be a knot represented as the closure of an $n$-braid $\beta$, and $\alpha \in \{2,3,\ldots\}$. Then
\[ \sup_{|q|=1} \log |Q^{\sltwo;V_{\alpha}}_{K}(q)| \leq n\log \alpha + \log \rho(\varphi_{\alpha}(\beta)) \leq
n\log \alpha + \frac{n(\alpha-1)}{2}h(\beta). \]
\end{theorem}
\begin{proof}
By definition of the spectral radius, 
\[ |Q^{\sltwo;V_{\alpha}}_{K}(q)| = | \tr(q^{\frac{H}{2}}\varphi_{\alpha}(\beta))| \leq \alpha^{n} \rho (q^{\frac{H}{2}}\varphi_{\alpha}(\beta)) \leq \alpha^{n} \rho(q^{\frac{H}{2}}) \rho(\varphi_{\alpha}(\beta)). \]
Here the last inequality follows from the fact that $q^{\frac{H}{2}}$ and $\varphi_{\alpha}(\beta)$ commutes. When $|q|=1$, 
$\rho(q^{\frac{H}{2}}) = 1$ hence by Theorem \ref{theorem:qestimate} (3), we conclude 
\[ \sup_{|q|=1} | \tr(q^{\frac{H}{2}}\varphi_{\alpha}(\beta))| \leq \sup_{|q|=1}\alpha^{n} \rho(\varphi_{\alpha}(\beta)) \leq \alpha^{n}e^{\frac{n(\alpha-1)}{2} h(\beta)}. \]
\end{proof}

By an analogy of the famous volume conjecture \cite{ka,mumu}, 
It is interesting to look at the asymptotic behavior of $|Q^{\sltwo;V_{\alpha}}_{K}(q)|$. 
By Theorem \ref{theorem:JvsE},
we have
\[  \frac{\sup_{|q|=1} \log |Q^{\sltwo;V_{\alpha}}_{K}(q)|}{\alpha} \leq n \frac{\log \alpha}{\alpha} + \frac{\sup_{|q|=1}\log \rho(\varphi_{\alpha}(\beta))}{\alpha} \leq 
n \frac{\log \alpha}{\alpha} + \frac{n(\alpha-1)}{2\alpha} h(\beta)
\]
This shows

\begin{equation}
\label{eqn:conj}
\limsup_{\alpha \to \infty} \frac{\sup_{|q|=1} \log |Q^{\sltwo;V_{\alpha}}_{K}(q)|}{\alpha} \leq  \limsup_{\alpha \to \infty} \frac{\sup_{|q|=1} \log \rho(\varphi_{\alpha}(\beta))}{\alpha} \leq \frac{n}{2} h(\beta).
\end{equation}

It is interesting to ask the convergence of the limits and when the inequalities (\ref{eqn:conj}) yield the equalities.
In particular, the second inequlaity is related to the question when the estimation of entropy from quantum representation is asymptotically sharp.

\renewcommand{\thesection}{\Alph{section}}
\setcounter{section}{0}

\section{Appendix: Multiforks for Lawrence's representation $L_{n,m}$}

In this appendix, we present multiforks in Lawrence's representation $L_{n,m}$ and explicit matrices of $L_{n,m}(\sigma_{i})$. For the basics of geometric treatments of Lawrence's representation, see \cite[Section 2]{i}.

First we review the definition of multiforks and how multifork represent a homology class in $H^{lf}_{m}(\widetilde{C_{m,n}};\Z)$. Let $Y$ be the $Y$-shaped graph with four vertices $c,r,v_1,v_2$ and oriented edges as shown in Figure \ref{fig:multiforks}(1).
A {\em fork} $F$ based on $d \in \partial D_{n}$ is an embedded image of $Y$ into $D^2=\{z \in \C \: | \: |z| \leqslant n+1\}$ such that:
\begin{itemize}
\item All points of $Y\setminus \{r,v_1,v_2\}$ are mapped to the interior of~$D_n$.
\item The vertex $r$ is mapped to $d_{i}$.
\item The other two external vertices $v_{1}$ and $v_{2}$ are mapped to the puncture points.
\end{itemize}


The image of the edge $[r,c]$ and the image of $[v_1,v_2]=[v_{1},c] \cup [c,v_{2}]$ regarded as a single oriented arc, are denoted by $H(F)$ and $T(F)$. We call $H(F)$ and $T(F)$ the \emph{handle} and the \emph{tine} of the fork $F$, respectively.

A {\em multifork} of dimension $m$ is an ordered tuples of $m$ forks $\bF= (F_{1},\ldots,F_{m})$ such that 
\begin{itemize}
\item $F_{i}$ is a fork based on $d_{i}$.
\item $T(F_{i}) \cap T(F_{j}) \cap D_{n} = \emptyset$ $(i\neq j)$.
\item $H(F_{i}) \cap H(F_{j}) = \emptyset$ $(i \neq j)$.
\end{itemize}
Figure \ref{fig:multiforks} (2) illustrates an example of a multifork of dimension $3$. We often use to represent multiforks consisting of $k$ parallel forks by drawing single fork labelled by $k$, as shown in Figure \ref{fig:multiforks} (3).  

\begin{figure}[htbp]
\centerline{\includegraphics[bb=119 618 489 718, width=120mm]{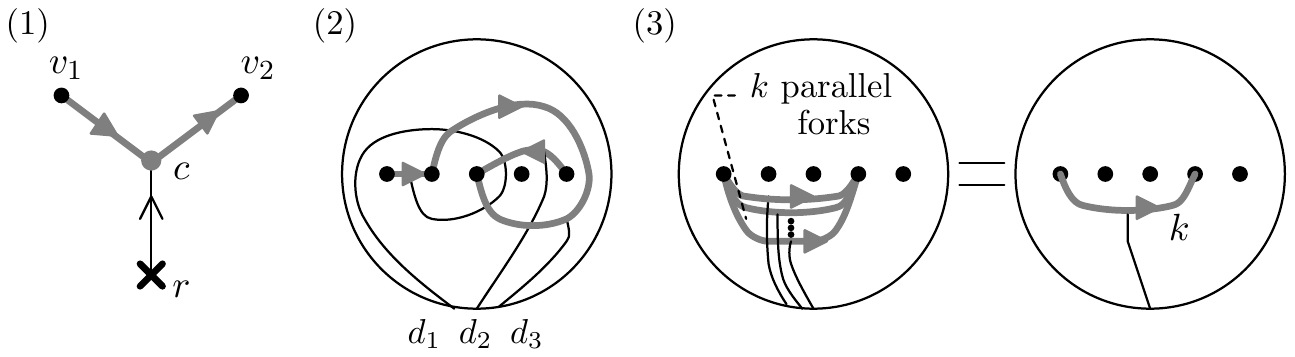}}
\caption{Multiforks: to distinguish tines and handle, we often write tine of forks by a bold gray line. }
 \label{fig:multiforks}
\end{figure}

For a multifork $\bF$, we regard the handle $H(F_{i})$ of the fork $F_{i}$ as a path $\gamma_{i} \co [0,1] \rightarrow D_{n}$ in $D_{n}$ by taking an appropriate parametrization. 
Then the handles of $\bF$ defines a path $H(\bF) = \{\gamma_{1},\ldots,\gamma_{m}\} \co [0,1] \rightarrow C_{n,m}$ in $C_{n,m}$.
Take a lift of $H(\bF)$, $\widetilde{H(\bF)} \co [0,1] \rightarrow \tC_{n,m}$ 
so that $\widetilde{H(\bF)}(0)=\widetilde{\mathbf{d}}$.

Let $\Sigma(\bF) = \left\{ \{z_{1},\ldots,z_{m} \} \in C_{n,m} \: | \: z_{i} \in T(F_{i}) \right\}$, and $\widetilde{\Sigma}(\bF)$ be the $m$-dimensional submanifold of $\tC_{n,m}$ which is the connected component of $\pi^{-1}(\Sigma(\bF))$ containing $\widetilde{H(\bF)}(1)$. The submanifold $\widetilde{\Sigma}(\bF)$ defines an element of $H_{m}^{lf}(\tC_{n,m};\Z)$. By abuse of notation, we will use $\bF$ to represent both multifork and the homology class $[\widetilde{\Sigma}(\bF)] \in H_{m}^{lf}(\tC_{n,m};\Z)$.

Here the orientation of $\widetilde{\Sigma}(\bF)$ is defined so that a canonical homeomorphism $T(F_1) \times \cdots \times T(F_m) \rightarrow \Sigma(\bF)$ is orientation preserving. Thus, for a fork $\bF_{\tau} = (F_{\tau(1)}, \cdots, F_{\tau(m)})$ obtained by permuting its coordinate by a permutation $\tau \in S_m$, we have $\bF_{\tau} = \textsf{sgn}(\tau) \bF \in H_{m}^{lf}(\tC_{n,m};\Z)$.

For $\be = (e_{1},\ldots,e_{n-1}) \in E_{n,m}$, we assign a multifork $\bF_{\be}=\{F_{1},\ldots, F_{m}\}$ in Figure \ref{fig:standmultiforks} and call $\bF_{\bf}$ a {\em standard multifork}. 

\begin{figure}[htbp]
\centerline{\includegraphics[bb=191 614 388 710,width=60mm]{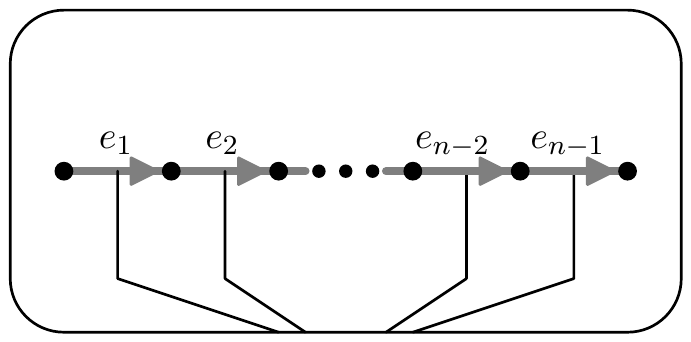}}
\caption{Standard multifork $\bF_{\be}$ for $\be=(e_1,\ldots,e_{n-1})$}
 \label{fig:standmultiforks}
\end{figure}

The set of standard multiforks spans a $\Z[x^{\pm 1},d^{\pm 1}]$-submodule $\mathcal{H}_{n,m}$ of $H_{m}^{lf}(\widetilde{C_{n,m}};\Z)$, which is free of 
dimension $\binom{n+m-2}{2}$ and is invariant under the $B_{n}$-action. This defines a (geometric) Lawrence's representation $L_{n,m} :B_{n} \rightarrow \GL \left(\binom{n+m-2}{2};\Z[x^{\pm 1},d^{\pm 1}]\right).$

From the definition of the submanifold $\widetilde{\Sigma}(\bF)$, 
we graphically express several relations among homology classes represented by multifork, which allows us to express a given multifork as a sum of standard multiforks (see \cite{kra,big} for the case $m=2$). 

\begin{figure}[htbp]
\centerline{\includegraphics[bb=117 495 482 717, width=120mm]{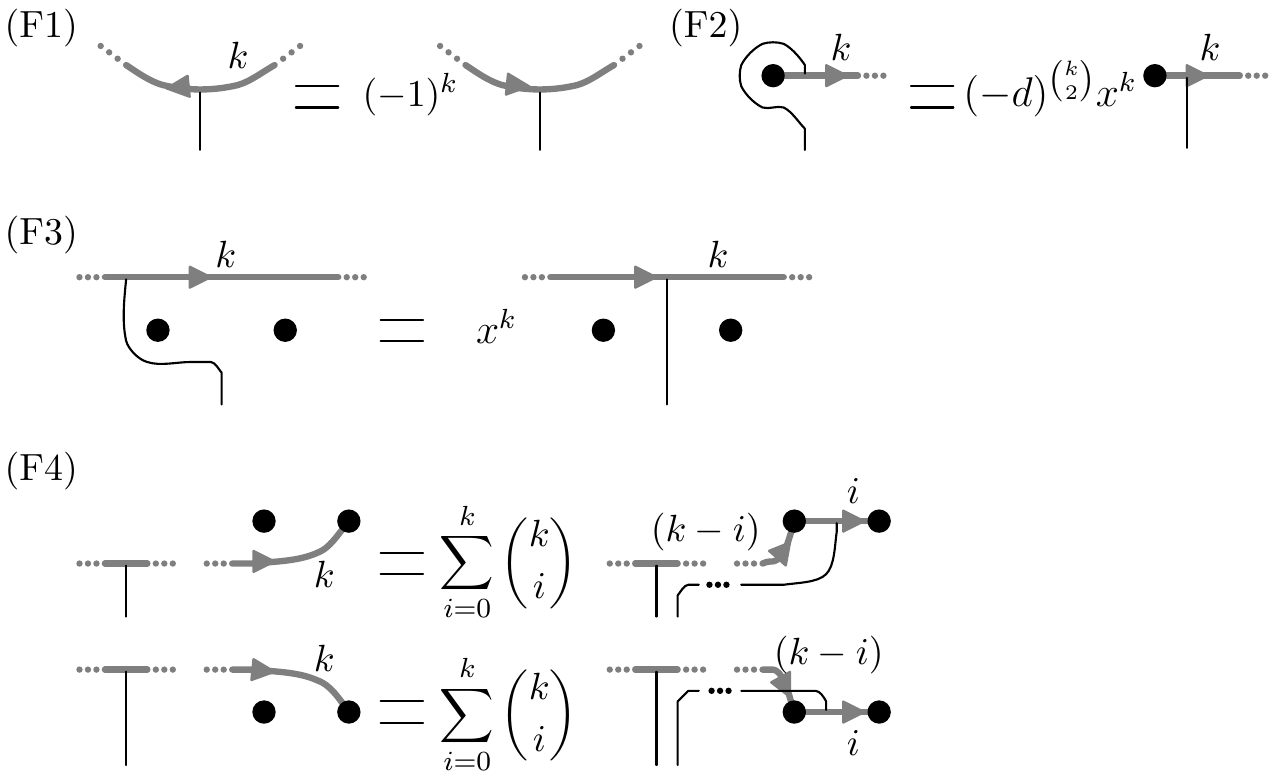}}
\caption{Geometric rewriting formula for multiforks}
 \label{fig:forkrule}
\end{figure}

In particular, these formulae leads to a formula of an explicit matrix representative of $L_{n,m}(\sigma_i)$:

\begin{equation}
\label{eqn:formula}
\begin{aligned}
 L_{n,m}(\sigma_{1})(\bF_{e_1,\ldots,e_{n-1}})& = 
\sum_{l=0}^{e_2} (-1)^{e_1} (-d)^{\binom{e_1}{2}} x^{e_1}\binom{e_2}{l} \bF_{e_1+e_{2}-l,l,\ldots} 
\\
 L_{n,m}(\sigma_{i})(\bF_{e_1,\ldots,e_{n-1}}) &= 
\sum_{k=0}^{e_{i-1}}\sum_{l=0}^{e_{i+1}} (-1)^{e_{i}} (-d)^{\binom{e_{i}+k}{2}} x^{e_{i} + k }\binom{e_{i-1}}{k}\binom{e_{i+1}}{l} \bF_{\ldots, e_{i-1}-k,
 e_{i}+k+e_{i+1}-l,l,\ldots} \\
 & \hspace{4cm}(i=2,\ldots,n-2)
\\
 L_{n,m}(\sigma_{n-1})(\bF_{e_1,\ldots,e_{n-1}}) &= 
\sum_{k=0}^{e_{n-2}} (-1)^{e_{n-1}} (-d)^{\binom{e_{n-1}+k}{2}} x^{e_{n-1} + k }\binom{e_{n-2}}{k} \bF_{\ldots,
 e_{n-2}-k,e_{n-1}+k} \\
\end{aligned}
\end{equation}
As we already mentioned, Proposition \ref{prop:reduced} follows from formula (\ref{eqn:formula}).

However, a multifork expression gives a direct way to see Proposition \ref{prop:reduced}: 
First note that by orientation convention of  $\widetilde{\Sigma}(\bF)$, when $d=-1$, the homology class represented by a multifork $\bF=(F_1,\ldots,F_m)$ is independent of a choice of indices of forks, namely, for a fork $\bF_{\tau} = (F_{\tau(1)}, \ldots, F_{\tau(m)})$ obtained by permuting its coordinate by a permutation $\tau \in S_m$, $\bF_{\tau} = \bF$. Therefore, the correspondence between multifork $\bF=(F_1,\ldots,F_m)$ that represents an element of $\mathcal{H}_{n,m}$ and a family of $m$ forks $\{F_1,\ldots,F_k\}$ that represents an element of $\Sym^{m} \mathcal{H}_{n,1}$ gives rise to the desired isomorphism $\mathcal{H}_{n,m} \rightarrow \Sym^{m} \mathcal{H}_{n,1}$.

\end{document}